\documentclass[12pt,reqno]{amsart}
\usepackage{amssymb}
\usepackage{txfonts}
\usepackage{amsmath}
\usepackage{multirow}
\usepackage{geometry}
\usepackage{enumerate}
\usepackage{comment}

 \newtheorem{thm}{Theorem}[section]
 \newtheorem{cor}[thm]{Corollary}
 \newtheorem{lem}[thm]{Lemma}

 \theoremstyle{definition}
 \newtheorem{defn}[thm]{Definition}
 \theoremstyle{remark}
 \newtheorem{rem}[thm]{Remark}
  
 \numberwithin{equation}{section}

\emergencystretch=\maxdimen
\begin{document}

\title[]
{No breather theorems for the mean curvature flow}

\author{Liang Cheng and Yongjia Zhang}

\dedicatory{}
\date{}



\thanks{Liang Cheng's Research partially supported by Natural Science Foundation of Hubei 2019CFB511
}

\address{School of Mathematics and Statistics $\&$ Hubei Key Laboratory of Mathematical Sciences, Central  China Normal University, Wuhan, 430079, P.R.China
}

\email{chengliang@mail.ccnu.edu.cn}

\address{ School of Mathematics, University of Minnesota, Twin Cities, MN, 55414, USA}
\email{ zhan7298@umn.edu}

\maketitle

\begin{abstract}
 In this article we study the breathers of the mean curvature flow in the Euclidean space. A breather is a solution to the mean curvature flow which repeats itself up to isometry and scaling once in a while. We prove several no breather theorems in the noncompact category, that is, under certain conditions, a breather of the mean curvature flow must be a solitonic solution (self-shrinker, self-expander, or translator).
\end{abstract}

\section{Introduction}

A breather of the mean curvature flow is a solution which is self-similar at two different times. Precisely, it is defined as follows.
\begin{defn}\label{def_breather}
Let $x:M^n\times I\to \mathbb{R}^{n+m}$ be a smooth and immersed solution to the mean curvature flow. If there exist two time instances $t_1,t_2\in I$, $t_1<t_2$, a constant $\alpha>0$, an isometry $F:\mathbb{R}^{n+m}\to \mathbb{R}^{n+m}$, and a self-diffeomorphism $\phi: M^n\rightarrow M^n$, such that
\begin{eqnarray}\label{eq:definitionofbreather}
 x(p,t_2)=\alpha F\circ x(\phi(p),t_1)\quad\text{ for all }\quad p\in M,
\end{eqnarray}
then $x$ is called a  breather. If $\alpha=1$, $\alpha<1$,  or $\alpha>1$, then the  breather is called  steady, shrinking, or expanding, respectively.
\end{defn}

Clearly, the \emph{solitonic solutions} (by this term we specifically mean the \emph{self-shrinkers}, the \emph{self-expanders}, and the \emph{translators}) can be regarded as the most trivial examples of the breather. Our goal in this article is to prove some no breather theorems for the mean curvature flow, that is,  under certain conditions, the breathers must be these trivial ones---the solitonic solutions. 

The no breather theorems were first proved by Perelman \cite{P1} for the Ricci flow. He proved that a closed breather of the Ricci flow must be a gradient Ricci soliton; these results are the first applications of his entropy formula. Later developments have generalized the no breather theorems to the noncompact Ricci flow under diverse conditions; see, for instance, \cite{CZhang}, \cite{CZhang1}, \cite{LZ}, \cite{RV}, \cite{Zhang1}, \cite{Zh}. Among them the strongest one for the shrinking case is proved by the authors \cite{CZhang}, where they have shown that a shrinking breather with Ricci curvature bounded from below must be a shrinking gradient Ricci soliton. The result in \cite{CZhang} is recently generalized to the harmonic Ricci flow in \cite{CC}.
Since the mean curvature flow is also a parabolic geometric evolution equation bearing many similar properties as the Ricci flow, it is natural to ask whether the same results also hold for the mean curvature flow. One may immediately think that Huisken's monotonicity formula \cite{Hu} would have some implications, and this is indeed how Magni-Mantegazza \cite[Theorem 2.3]{MM} proved that a closed shrinking breather of the mean curvature flow must be a self-shrinker. Applying a method similar to that of Perelman, they considered the supremum of Huisken's entropy, which is always attained in the closed case; see \cite{MM} or Remark \ref{remarkmm} for a brief discussion of this method. In the closed case,  Magni-Mantegazza also proved the no steady and expanding breather theorems for the closed mean curvature flow. Their argument goes as follows. A closed mean curvature flow (with any codimension) always develops finite-time singularity. But a steady or expanding breather gives rise to an immortal solution; see section 4 and section 5 below. Hence closed steady or expanding breather does not exist. However, Magni-Mantegazza's methods can not be applied to the noncompact case.  In this paper, we will prove some no steady and expanding breather theorems in the noncompact case.

We remark that our study looks more interesting in light of the examples of general self-similar mean curvature flows. Recall that a self-similar
solution is a mean curvature flow $x:M^n\times I\to\mathbb{R}^{n+m}$ satisfying
\begin{eqnarray*}\label{eq:definitionofself-similarsolution}
	x(M,t)=\alpha(t) F(t)\circ x(M)\quad \text{ for all }\quad t\in I,
\end{eqnarray*}
where $\alpha(t)$ is a positive function  with $\alpha(0)=1$
and $F(t)$ is a one-parameter family of isometries in $\mathbb{R}^{n+m}$ with $F(0)=\operatorname{id}$.  Note that except for the  solitonic solutions, there are other self-similar solutions such as rotators and those of mixed types (they rotate while at the same time shrink/expand). The self-similar solutions to the curve shortening flow in $\mathbb{R}^{2}$ are classified by Halldorsson \cite{H}. These examples show that one can not hope the no breather theorem to hold without any assumption; we shall include them as counterexamples in the remarks after our main theorems.

As our first main result, we state the no shrinking breather theorem, since, among all of our results, its conclusion is the most strong. \emph{In the statements of Theorem \ref{main}---\ref{main5}, the time instances $t_1$ and $t_2$ correspond to the ones in Definition \ref{def_breather}.}

\begin{thm}\label{main}
Let $x:M^n\times [t_1,t_2]\to \mathbb{R}^{n+m}$, be a  complete immersed shrinking breather of the mean curvature flow. Let $\alpha\in(0,1)$ be the constant in (\ref{eq:definitionofbreather}). Furthermore, assume
\begin{equation}\label{integral_bound}
\int_{M} e^{-\gamma|x(\cdot,t_1)|^2}d\mu_{t_1}<\infty,
\end{equation}
where $\gamma$ is an arbitrary constant satisfying $\gamma<\tfrac{1}{4(t_2-t_1)}(1-\alpha^2)$ and $\mu_t$, where $t\in[t_1,t_2]$, is the Riemannian measure induced by the immersion $x(\cdot,t):M^n\rightarrow \mathbb{R}^{n+m}$. Then $x$ must be a self-shrinker. 
\end{thm}

\begin{rem}\label{main_rem}
\begin{enumerate}
\item There exist self-similar solutions to the curve shortening flow which rotate and shrink at the same time (see \cite[Figures 12---19]{H}). None of these examples satisfies (\ref{integral_bound}), since they all wind infinitely many times near a circle with finite radius. This shows that the condition (\ref{integral_bound}) can not be removed.

\item Theorem \ref{main} can also be proved for the \emph{Brakke flow breathers}. The details are left to the readers; note that the condition (\ref{integral_bound}) is necessary. A Brakke flow breather should be defined as follows. Let $\mu_t$, where $t\in I$, be an $n$-dimensional Brakke flow in $\mathbb{R}^{n+m}$. Then, $\mu_t$ is a Brakke flow breather, if there exists $t_1<t_2$, $\alpha>0$, and isometry $F:\mathbb{R}^{n+m}\to\mathbb{R}^{n+m}$ as in Definition \ref{def_breather}, such that
$$\mu_{t_2}=\alpha\cdot F_{\#}\mu_{t_1},$$
where $F_{\#}\mu$ stands for the push-forward of $\mu$ by $F$.

\item If there is no translation part in the isometry $F$ in (\ref{eq:definitionofbreather}), that is, if $F$ can be written as $F(x)=\mathcal{R} x$ for some $\mathcal{R}\in O(n+m)$, then Theorem \ref{main} also holds if the constant $\gamma=\tfrac{1}{4(t_2-t_1)}(1-\alpha^2)$; one may check this fact with the proof of Lemma \ref{firstthereisabound} and Lemma \ref{integralalwaysbound}.
\end{enumerate}
\end{rem}

Next, we state the no steady and no expanding breather theorems. The examples in \cite{H} indicate that (\ref{integral_bound}) is not sufficient for a no steady or expanding breather theorem. The examples in \cite[Figures 8---9]{H} are curve shortening flows evolving by rotation, and the examples in \cite[Figures 10---11]{H} evolve by rotation along with expansion. The reader can easily check that all these examples satisfy (\ref{integral_bound}). Hence, in the noncompact case, the no steady or expanding breather theorems are not expected to hold unless under strong conditions.

\begin{thm}\label{main3}
	Let $x_0:M^n \times[t_1,t_2]\to \mathbb{R}^{n+1}$ be a complete, noncompact, and weakly convex expanding breather of the mean curvature flow with bounded second fundamental form. Then $x_0$ is a self-expander.
\end{thm}


\begin{thm}\label{main5}
	Let $x_0:M^n\times[t_1,t_2]\to \mathbb{R}^{n+1}$ be a complete, noncompact, and weakly convex steady breather of the mean curvature flow with bounded second fundamental form.  Moreover, assume \textbf{either} one of the following is true.
	\begin{enumerate} [(1)]
	\item There exists a point $p_0\in M$ such that $\displaystyle \max_{M}H(\cdot,t_1)$ is attained at $p_0$.
	\item There exists a point $p_0\in M$ such that
	\begin{equation}\label{bound_condition}
	\text{either }\quad\sup_{j\in\mathbb{N}}\left|\,  x(\phi^{-j}(p_0),t_1) \, \right|<\infty\quad \text{ or }\quad \sup_{j\in\mathbb{N}}\left|\, x(\phi^{j}(p_0),t_1) \, \right|<\infty.
	\end{equation}
	\end{enumerate}
	Then $x_0$ must be a translator. 
\end{thm}

\begin{rem}
	The examples in \cite[Figures 10---11]{H} evolve by rotation along with expansion. They are not weakly convex.  This shows that without the weakly convex condition, Theorem \ref{main3} does not hold.
	The Altschuler's yin-yang spiral in \cite{Al} is a rotator satisfying (\ref{bound_condition}); to see this, one needs only to observe that this curve passes through the origin, and then check the definition of rotator in section 2; see (\ref{rotator_generated_vector}). However, it is not weakly convex. This shows that without the weakly convex condition, Theorem \ref{main5}(2) does not hold.
\end{rem}

An immediate application of Theorem \ref{main} and Theorem \ref{main3} is the following
corollary.
\begin{cor}\label{cor_main}
There exists no complete self-similar solution to the mean curvature flow satisfying (\ref{integral_bound}) which rotates and shrinks at the same time. Moreover, there exists no complete weakly convex self-similar solution to the mean curvature flow with bounded second fundamental form which rotates and expands at the same time. 
\end{cor}

Though Theorem \ref{main5} has a corollary similar to Corollary \ref{cor_main}, we have a stronger result  based on a more direct proof.

\begin{thm}\label{Coro_main5}
Assume that $x_0:M^n\to\mathbb{R}^{n+1}$ is a  properly immersed weakly mean convex rotator. Then $x_0$ must be a minimal hypersurface. If we further assume $x_0$ to be weakly convex, then it must be a hyperplane.
\end{thm}

Finally, let us recall the statements of the different versions of no breather theorems for the Ricci flow. All of them, except for the diverse conditions, are formulated in almost the same way, namely, a breather, under certain conditions, must be the critical point of certain monotonicity formula, ergo a gradient Ricci soliton. In other words, under these conditions all breathers must be ``trivial''. Therefore, it remains interesting to find examples of nontrivial breathers. Obviously, there are ``semi-nontrivial'' examples---the non-gradient solitons; see \cite{L} for an example of complete non-gradient expanding soliton. The canonical form of a non-gradient Ricci soliton flows by self-diffeomorphism and hence satisfies the definition of the breather, but it is not a gradient soliton, that is, a critical point of a monotonicity formula (Perelman's entropy or Hamilton's Harnack). Nevertheless, the non-gradient solitons still move by self-diffeomorphism. Recently, Topping \cite{T} has constructed a ``truly'' nontrivial expanding breather of the Ricci flow, that is, he has proved there exists an expanding breather of the Ricci flow which is not the Ricci soliton. As we have seen above, the same issue arises when studying the breathers of the mean curvature flow. It remains interesting to ask whether there exists non-self-similar mean curvature flow breather, namely, mean curvature flow solution with nontrivial periodicity. We conjecture that there at least exists a non-self-similar expanding breather of the mean curvature flow as in the case of the Ricci flow.

\section{Preliminaries}

\subsection{Huisken's monotonicity formula}

Huisken's monotonicity formula was originally defined for compact mean curvature flows. Nevertheless, there are multiple ways to generalize it to the noncompact case. We shall include a version that is proper to our end. First of all, Let us define the backward heat kernel $\Phi_{x_0,t_0}:\mathbb{R}^{n+m}\times(-\infty,t_0)\rightarrow\mathbb{R}$ as 
\begin{eqnarray}\label{BHK}
\Phi_{x_0,t_0}(x,t):=(4\pi(t_0-t))^{-\frac{n}{2}}e^{-\frac{|x-x_0|^2}{4(t_0-t)}}.
\end{eqnarray}

\begin{thm}\label{Huiskenmonotonicity}
Let $x:M^n\times I\rightarrow\mathbb{R}^{n+m}$ be a smooth and complete immersed solution to the mean curvature flow. Given $(x_0,t_0)\in \mathbb{R}^{n+m}\times\mathbb{R}$, let $t_1,t_2\in(0,t_0)\cap I$ with $t_1<t_2$. If
$$
\int_M\Phi_{x_0,t_0}(x,0)d\mu_0<\infty,
$$
then we have
$$
\int_M\Phi_{x_0,t_0}(x,t)d\mu_t\le \int_M\Phi_{x_0,t_0}(x,0)d\mu_0<\infty
$$
for all $t>0$ under the mean curvature flow.
Moreover we have the following Huisken's monotonicity holds
\begin{align}\label{eqmonotonicityhuisken}
&\int_M \Phi_{x_0,t_0}(x(\cdot,t_2),t_2)d\mu_{t_2}-\int_M \Phi_{x_0,t_0}(x(\cdot,t_1),t_1)d\mu_{t_1}
\\\nonumber
= \ &-\int_{t_1}^{t_2}\int_M \left|\vec{H}+\frac{(x(\cdot,t)-x_0)^\perp}{2(t_0-t)}\right|^2\Phi_{x_0,t_0}(x(\cdot,t),t)d\mu_t dt,
\end{align}
where $\Phi_{x_0,t_0}$ is the backward heat kernel defined in (\ref{BHK}) and $\mu_t$ stands for the Riemannian measure induced by the immersion $x(\cdot,t):M\rightarrow\mathbb{R}^{n+m}$.
\end{thm}
\begin{proof}
	In order to avoid the integration by parts which may depend on extra conditions, 
 we consider the following drifted mean curvature flow as in \cite{CS} 
\begin{equation}\label{drifting_mcf}
\partial_t  x=\vec H-\frac{(x-x_0)^\top}{2(t_0-t)},
\end{equation}
Notice that (\ref{drifting_mcf}) differs from the original mean curvature flow only by the tangential tangential flow
generated by the vector fields $-\frac{(x-x_0)^\top}{2(t_0-t)}$.
We denote by $\{\nu_{\alpha}\}_{\alpha=n+1}^{n+m}$ an orthonormal frame of the normal bundle.
Then we may compute
\begin{align*}
	\frac{\partial g_{i j}}{\partial t} &=2 \partial_{i}\left(\vec H-\frac{(x-x_0)^\top}{2(t_0-t)}\right) \cdot \partial_{j} x \\
	&=-2\vec H \cdot \vec h_{i j}-\frac{1}{t_0-t} \partial_{i}\left((x-x_0)-\sum\limits_{\alpha}\langle x-x_0,\nu^{\alpha}\rangle \nu^{\alpha}\right) \partial_{j} x \\
	&=-2\vec H \cdot \vec h_{i j}-\frac{1}{t_0-t} g_{i j}-\frac{1}{t_0-t}\sum\limits_{\alpha}\langle x-x_0,\nu^{\alpha}\rangle  h^{\alpha}_{i j} ,
\end{align*}
where $h^{\alpha}_{ij}=\nu_{\alpha}\cdot \partial_{i}\partial_{j} x$, $\vec h_{i j}=h^{\alpha}_{ij}\nu_{\alpha}$, and $\vec H=g^{ij}\vec h_{ij}$.
It follows that
\begin{align*}
	\frac{\partial }{\partial t} d\mu_t=\left(-|\vec H|^2-\frac{n}{2(t_0-t)}-\frac{1}{2(t_0-t)}\langle x-x_0,\vec H\rangle\right)  d\mu_t  .
\end{align*}
Moreover, we have
\begin{align*}
\frac{\partial}{\partial t}\Phi_{x_0,t_0}( x(\cdot,t),t) =\left(-\frac{|(x-x_0)^{\perp}|^2}{4(t_0-t)^2}+\frac{n}{2(t_0-t)}-\frac{1}{2(t_0-t)}\langle x-x_0,\vec H\rangle\right)\Phi_{x_0,t_0}( x(\cdot,t),t). 
\end{align*}
Hence
we have
\begin{align*}
	\frac{\partial}{\partial t}(\Phi_{x_0,t_0}( x(\cdot,t),t)d\mu_t) =-\left|\vec{H}+\frac{(x(\cdot,t)-x_0)^\perp}{2(t_0-t)}\right|^2\Phi_{x_0,t_0}(x(\cdot,t),t)d\mu_t\leq 0. 
\end{align*}
Note that the  flow (\ref{drifting_mcf}) is only a reparametrization of the original one, which does not affect the computation of Huisken's functional. The conclusion then follows from Tonelli's theorem.
\end{proof}

\subsection{Self-similar solutions}

We briefly introduce different kinds of self-similar solutions to the mean curvature flow. For a more detailed treatment, refer to \cite[Section 6.2]{ACGL}.

\emph{Translators.}  If $x_0:M^n\to\mathbb{R}^{n+m}$ is an immersed submanifold satisfying 
$$\vec{H}=e^{\perp},$$
where $\vec{H}$ is the mean curvature vector and $e\in\mathbb{R}^{n+m}$ is a fixed unit vector, then this immersion generates a mean curvature flow moving by translation.  In other words, $x:M^n\times\mathbb{R}\to\mathbb{R}^{n+m}$ defined by
$$x(p,t):=x_0(\phi(p,t))+et,$$
where $\phi$ is the flow generated by $V:=-dx_0^{-1}(e^{\top})$, evolves by the mean curvature. Here and below, we use $dx_0: \Gamma(TM)\to C^\infty(M,\mathbb{R}^{n+m})$ to denote the push-forward induced by the immersion $x_0$.

\emph{Self-shrinkers.} Let $x_0: M^n\to\mathbb{R}^{n+m}$  be an immersed submanifold satisfying
$$\vec{H}=-\frac{1}{2}x_0^\perp,$$
then $x:M^n\times(-\infty,0)\to\mathbb{R}^{n+m}$ defined by
$$x(p,t):=\sqrt{-t}x_0\left(\phi\left(p,\tfrac{1}{2}\log(-t)\right)\right),$$
where $\phi$ is the flow generated by $V:=-dx_0^{-1}(x_0^\top)$, is a solution to the mean curvature flow. It is well known that a self-shrinker is a critical point of Huisken's monotonicity formula.

\emph{Self-expanders.} Self-expanders are similar to self-shrinkers, since their definitions differ only by a sign. Let $x_0: M^n\to\mathbb{R}^{n+m}$  be an immersed submanifold satisfying
$$\vec{H}=\frac{1}{2}x_0^\perp,$$
then $x:M^n\times(0,\infty)\to\mathbb{R}^{n+m}$ defined by
$$x(p,t):=\sqrt{t}x_0\left(\phi\left(p,\tfrac{1}{2}\log(t)\right)\right),$$
where $\phi$ is the flow generated by $V:=-dx_0^{-1}(x_0^\top)$, is a solution to the mean curvature flow.

\emph{Rotators.} Let $x_0: M^n\to\mathbb{R}^{n+m}$ be an immersed submanifold and $J\in\mathfrak{so}(n+m)$, satisfying
$$\vec{H}=(Jx_0)^\perp,$$
then $x:M^n\times\mathbb{R}\to\mathbb{R}^{n+m}$ defined by 
$$x(p,t):=e^{tJ}\cdot x_0(\phi(p,t)),$$
where $\phi$ is the flow generated by
\begin{equation}\label{rotator_generated_vector}
	 V:=-dx_0^{-1}\left((Jx_0)^\top\right),
\end{equation}
 is a rotating self-similar solution to the mean curvature flow.

\emph{Mixed type.} There also exist self-similar solutions which evolve by more complex self-similarities, for instance, rotation together with shrinking/expansion. For these examples, refer to \cite{H}.

\section{Proof of Theorem \ref{main}}
The idea of proving Theorem \ref{main} is similar as that which was applied in \cite{CZhang} and \cite{Zh}. We will first of all implement the method in \cite{LZ} to construct an ancient solution using the shrinking breather, by splicing properly scaled and properly located copies of the breather head-to-tail. Then, this ancient solution admits a blow-down limit, which must be a critical point of Huisken's monotonicity formula, and hence a self-shrinker. Finally, because of the construction of the ancient solution, the limit must be identical to the original breather up to a scale and an isometry. This is the idea of the proof.

After parabolic scaling and translating in time, we let $x:M^n\times[0,1]\to \mathbb{R}^{n+m}$, where $t_1=0$ and $t_2=1$, be the breather in the statement of Theorem \ref{main}. Let $\alpha\in(0,1)$ be the constant therein. According to the definition of breather, we can find an isometry $F:\mathbb{R}^{n+m}\to \mathbb{R}^{n+m}$ and a self-diffeomorphism $\phi:M\to M$, such that
$$
x(p,1)=\alpha F\circ x(\phi(p),0)\quad \text{ for all }\quad p\in M.
$$
For the sake of convenience, let $y_0(p,\tau)=x(p,t)$ with $\tau(t)=1-t$, and the backward time $\tau$ shall be considered instead throughout this section. Since  $F$ is an isometry on $\mathbb{R}^{n+m}$, there exist an orthogonal matrix $\mathcal{R}\in O(n+m)$  and  a constant vector $V\in\mathbb{R}^{n+m}$ such that
\begin{equation}\label{breather_1}
y_0(p,1)=F^{-1}\left(\alpha^{-1}  y_0\left(\phi^{-1}(p),0\right)\right)=\alpha^{-1} \mathcal{R} y_0\left(\phi^{-1}(p),0\right)+V\quad \text{ for all }\quad p\in M.
\end{equation}
For each $j\geq 0$ we define 
\begin{eqnarray}\label{y_j_definition}
\displaystyle \tau_j&=&\sum^j_{k=0}\alpha^{-2k},
\\\nonumber
y_j(p,\tau)&=&\alpha^{-j}\mathcal R^j y_0\left(\phi^{-j}(p),\alpha^{2j}( \tau- \tau_{j-1})\right)+\sum\limits_{k=0}^{j-1}\alpha^{-k}\mathcal R^k V,\ \  \tau\in [ \tau_{j-1}, \tau_j].
\end{eqnarray}
Then we obviously have $\tau_j\nearrow \infty$. 
As mentioned at the beginning of the section,  we define the spliced ancient solution 
\begin{eqnarray}\label{defined_ancient_solution}
y(p ,\tau)=\left\{
\begin{array}{ll}
y_0(p ,\tau), \quad &  \tau\in [0,1], \\
y_j(p ,\tau), \quad & \tau\in [\tau_{j-1}, \tau_j].
\end{array}
\right.
\end{eqnarray}

	\begin{lem}\label{TypeIancientsolution}
	$y(p ,\tau)$ defined in (\ref{defined_ancient_solution}) is a smooth ancient solution to the backward mean curvature flow.
\end{lem}
\begin{proof}
	Clearly we only need to check that $y(p ,\tau)$ is smooth at each $\tau_j$.
	Applying (\ref{breather_1}) and (\ref{y_j_definition}), we get
	\begin{align*}
	y_j(p,\tau_{j-1})&=\alpha^{-j}\mathcal R^j y_0\left(\phi^{-j}(p),0\right)+\sum\limits_{k=0}^{j-1}\alpha^{-k}\mathcal R^k V\\
	&
	=\alpha^{-(j-1)}\mathcal R^{j-1} \left(y_0\left(\phi^{-(j-1)}(p),1\right)-V\right)+\sum\limits_{k=0}^{j-1}\alpha^{-k}\mathcal R^k V\\
	&=\alpha^{-(j-1)}\mathcal R^{j-1} y_0\left(\phi^{-(j-1)}(p),\alpha^{2(j-1)}( \tau_{j-1}- \tau_{j-2})\right)+\sum\limits_{k=0}^{j-2}\alpha^{-k}\mathcal R^k V\\
	&=y_{j-1}(p,\tau_{j-1}).
	\end{align*}
By the definition of $y$, we have
	\begin{align*}
	\frac{\partial_+}{\partial\tau}y(p,\tau_{j-1})=\frac{\partial_+}{\partial\tau}y_j(p,\tau_{j-1})=\alpha^{j}\mathcal R^j\frac{\partial_+}{\partial \tau}y_0\left(\phi^{-j}(p),0\right)=-\alpha^{j}\mathcal R^j\vec{H}_0\left(\phi^{-j}(p),0\right),
	\end{align*}
	and 
	\begin{eqnarray*}
	\frac{\partial_-}{\partial \tau}
y(p,\tau_{j-1})&=&\frac{\partial_-}{\partial \tau}
y_{j-1}(p,\tau_{j-1})=\alpha^{j-1}\mathcal R^{j-1}\frac{\partial}{\partial \tau}y_0\left(\phi^{-(j-1)}(p),1\right)
\\
&=&-\alpha^{j-1}\mathcal R^{j-1}\vec{H}_0\left(\phi^{-(j-1)}(p),1\right)=-\alpha^{j}\mathcal R^j\vec{H}_0(\phi^{-j}(p),0).
	\end{eqnarray*}
	Hence, we have
	$$\frac{\partial_+}{\partial \tau}
	y(p,\tau_{j-1})=	\frac{\partial_-}{\partial \tau}
	y(p,\tau_{j-1}).
	$$
	Similarly, it is straightforward to check
	$$\frac{\partial^k_+}{\partial \tau^k}
	y(p,\tau_{j-1})=	\frac{\partial^k_-}{\partial \tau^k}
	y(p,\tau_{j-1})\quad\text{ for all }\quad k\geq 2.$$ This finishes the proof of the lemma.
\end{proof}

One important ingredient of the proof is to show that the ancient solution $y$ has bounded entropy. The first step is to show that this holds for the original breather $x$. For all $\varepsilon\ge0$, we define the backward heat kernel $\Phi_\varepsilon$ as 
\begin{eqnarray}\label{definitionofthekernel}
\tau_0&:=&(\alpha^{-2}-1)^{-1},
\\\nonumber
\Phi_\varepsilon(x,\tau)&:=&(4\pi(\tau+\tau_0+\varepsilon))^{-\frac{n}{2}}\exp\left(-\frac{|x|^2}{4(\tau+\tau_0+\varepsilon)}\right).
\end{eqnarray}

\begin{lem}\label{firstthereisabound}
There exists a positive constant $\bar\varepsilon$ depending only on the constant $\gamma$ in (\ref{integral_bound}), such that for all $\varepsilon\in[0,\bar\varepsilon)$, it holds that 
\begin{eqnarray}\label{boundedentropyin01}
\int_M \Phi_\varepsilon(y(\cdot,\tau),\tau)d\mu_{\tau}\le C\quad\text{ for all }\quad \tau\in[0,1],
\end{eqnarray} where $y$ is defined in (\ref{defined_ancient_solution}), $\mu_\tau$ is the Riemannian measure induced by $y(\cdot,\tau):M\rightarrow\mathbb{R}^{n+m}$, and $C$ is a constant depending only on the quantity in (\ref{integral_bound}). 
\end{lem}
\begin{proof}
By Theorem \ref{Huiskenmonotonicity}, we need only to show that (\ref{boundedentropyin01}) holds at $\tau=1$. Since $\gamma<\tfrac{1}{4}(1-\alpha^2)$, by (\ref{definitionofthekernel}), we have
$$\frac{1}{4(1+\tau_0+\varepsilon)}=\frac{1-\alpha^2}{4(1+\varepsilon(1-\alpha^2))}\geq\gamma,$$
if $\bar\varepsilon$ is taken to be small enough. It then follows from (\ref{integral_bound}) that (\ref{boundedentropyin01}) holds when $\tau=1$. The lemma then follows from Theorem \ref{Huiskenmonotonicity}.
\end{proof}

\begin{lem}\label{integralalwaysbound}
There is a constant $C$, such that
\begin{eqnarray}
\int_M \Phi(y(\cdot,\tau),\tau)d\mu_\tau\leq C\quad\text{ for all }\quad \tau> 0,
\end{eqnarray}
where $y$, $\mu_\tau$ are as defined before, and
\begin{eqnarray}\label{HKS}
\Phi(x,\tau):=\frac{1}{(4\pi\tau)^{\frac{n}{2}}}e^{-\frac{|x|^2}{4\tau}}.\end{eqnarray}
\end{lem}
\begin{proof}
Let us fix an arbitrary $\tau> 0$ and assume $\tau\in[\tau_{j-1},\tau_j]$ for some $j\in\mathbb{N}$. Without loss of generality, we may assume $j>0$, for otherwise we already have the bound (\ref{boundedentropyin01}) for $\tau\in(0,1]$. Then $y(\cdot,\tau)=y_j(\cdot,\tau)$, where the latter is defined in (\ref{y_j_definition}). Since a time-independent reparametrization does not affect the mean curvature flow equation or the integral, we may instead consider 
\begin{eqnarray*}
y_j(\phi^j(\cdot),\tau)&=&\alpha^{-j}\mathcal R^j y_0\left(\cdot,\alpha^{2j}( \tau- \tau_{j-1})\right)+\sum\limits_{k=0}^{j-1}\alpha^{-k}\mathcal R^k V
\\
&:=&\alpha^{-j}\left(\mathcal R^j y_0\left(\cdot,\tilde\tau\right)+v_i\right),
\end{eqnarray*}
where we have defined 
\begin{eqnarray}
\tilde\tau&:=&\alpha^{2j}( \tau- \tau_{j-1})\in[0,1],
\\\nonumber
v_i&:=&\alpha^j\sum\limits_{k=0}^{j-1}\alpha^{-k}\mathcal R^k V\in\mathbb{R}^{n+m}.
\end{eqnarray}
It is easy to verify by straightforward computation that $|v_j|\leq c_0$ for some constant $c_0$ independent of $j$. Hence, letting $\delta$ be a small positive constant to be fixed, we have
$$|y_j(\phi^j(\cdot),\tau)|^2\geq \alpha^{-2j}\Big((1-\delta)|y_0\left(\cdot,\tilde\tau\right)|^2-C_\delta\Big),$$ where $C_\delta$ is a constant depending on $c_0$ and $\delta$, but independent of $j$.

Finally, we have
\begin{align*}
&\int_M\Phi(y(\cdot,\tau),\tau)d\mu_\tau
\\
\leq\, &\int_M\frac{1}{(4\pi(\alpha^{-2j}\tilde\tau+\tau_{j-1}))^{\frac{n}{2}}}\exp\left(-\frac{\alpha^{-2j}\Big((1-\delta)|y_0\left(\cdot,\tilde\tau\right)|^2-C_\delta\Big)}{4(\alpha^{-2j}\tilde\tau+\tau_{j-1})}\right)\alpha^{-nj}d\mu_{\tilde\tau}
\\
\leq\, & C_\delta\int_M\frac{1}{(4\pi(\tilde\tau+(1-\alpha^{2j})\tau_0))^{\frac{n}{2}}}\exp\left(-\frac{(1-\delta)|y_0\left(\cdot,\tilde\tau\right)|^2}{4(\tilde\tau+(1-\alpha^{2j})\tau_0)}\right)d\mu_{\tilde\tau}
\\
\leq\, & C_\delta\int_M\frac{1}{(4\pi(\tilde\tau+\tau_0+\varepsilon(\delta)))^{\frac{n}{2}}}\exp\left(-\frac{|y_0\left(\cdot,\tilde\tau\right)|^2}{4(\tilde\tau+\tau_0+\varepsilon(\delta))}\right)d\mu_{\tilde\tau}
\\
=\, & C_\delta\int_M \Phi_{\varepsilon(\delta)}(y(\cdot,\tilde\tau),\tilde\tau)d\mu_{\tilde\tau}.
\end{align*}
Note that in the above computation the constant $C_\delta$ varies from line to line, and that $\varepsilon(\delta)$ depends only on $\delta$ and satisfies $\varepsilon(\delta)\searrow 0$ as $\delta\searrow 0$. Therefore, taking $\delta$ small enough such that $\varepsilon(\delta)<\bar\varepsilon$, the conclusion follows from (\ref{boundedentropyin01}).
\end{proof}

Let us scale the ancient solution $y:M\times[0,\infty)\to\mathbb{R}^{n+m}$ with the factors $\left\{\tau_j^{-\frac{1}{2}}\right\}_{j=1}^\infty$ and obtain the sequence of mean curvature flows $\{\tilde y_j(p,\tau)\}_{j=1}^\infty$. More precisely, we define 
$$\tilde y_j(p,\tau):= \tau_j^{-\frac{1}{2}}y(\phi^{j+1}(p),\tau_j\tau)\quad\text{ for  }\quad \tau\in \left[1,\tfrac{\tau_{j+1}}{\tau_j}\right],$$
where $y$ is defined in (\ref{defined_ancient_solution}). Note that the reparametrization does not affect anything. Then we have
\begin{equation}\label{estimate_1}
\tilde{y}_j(p,\tau)=
\tau_j^{-\frac{1}{2}}\alpha^{-(j+1)}\mathcal R^{j+1} y_0\left(p,\alpha^{2(j+1)}\tau_j(\tau-1)\right)+\tau_j^{-\frac{1}{2}}\sum\limits_{k=0}^{j}\alpha^{-k}\mathcal R^k V
\end{equation}
for all $\tau\in \left[1,\frac{\tau_{j+1}}{\tau_j}\right]$. Here we remark that
$$\frac{\tau_{j+1}}{\tau_j}\searrow\alpha^{-2}.$$ 
By the definition  (\ref{y_j_definition}) of $\tau_j$, we estimate
\begin{align}\label{estimate_3}
\left|\tau_j^{-\frac{1}{2}}\sum\limits_{k=0}^{j}\alpha^{-k}\mathcal R^k V\right|&\le \tau_j^{-\frac{1}{2}}\sum\limits_{k=0}^{j}\alpha^{-k} |V|\le c|V|,
\end{align}
where $c$ is a constant independent of $j$. Since the scaling factors in (\ref{estimate_1}) are all controlled, that is, 
$$\alpha^{2(j+1)}\tau_j\rightarrow c_0:=(\alpha^{-2}-1)^{-1}, $$
we have that $\tilde y_j$ and $y_0$ differ only by a rotation, a scaling with a controlled factor, and a transportation by a controlled vector. Next, since $O(n+m)$ is a compact Lie group, after passing to a subsequence, we have 
$$\mathcal R^{j+1}\rightarrow \mathcal R^\infty\in O(n+m),\quad \tau_j^{-\frac{1}{2}}\sum\limits_{k=0}^{j}\alpha^{-k}\mathcal R^k V\rightarrow V^\infty\in\mathbb{R}^{n+m}.$$ If we write
$$\tilde y_\infty\big(p,\tau)=c_0^{-\frac{1}{2}}\mathcal R^\infty y_0(p,c_0(\tau-1)\big)+V^\infty\quad\text{ for all }\quad (p,\tau)\in M\times[1,\alpha^2],$$
then it is trivial to see that $\tilde y_j\to \tilde y_\infty$ pointwise smoothly, that is,
$$\nabla^m\tilde y_j(p,\tau)\to \nabla^m\tilde y_\infty (p,\tau)\quad \text{ for all }\quad m\in\mathbb{N} \quad\text{ and }\quad (p,\tau)\in M\times[1,\alpha^2].$$
Now, it remains only to show that $\tilde y_\infty$ is a critical point of Huisken's monotonicity formula.

Let $\Phi:\mathbb{R}^{n+m}\times(0,\infty)\to\mathbb{R}$ be as defined in (\ref{HKS}). We shall use the notations $\tilde\mu_{j,\bar\tau}$ to denote the Riemannian measure induced by $\tilde y_j(\cdot,\bar\tau)$. For any $1\leq\bar\tau_1<\bar\tau_2\leq\alpha^{-2}$, we have
\begin{align*}
&\int_M\Phi(\tilde y_j(\cdot,\bar\tau_2),\bar\tau_2)d\tilde\mu_{j,\bar\tau_2}-\int_M\Phi(\tilde y_j(\cdot,\bar\tau_1),\bar\tau_1)d\tilde\mu_{j,\bar\tau_1}
\\
=\,&\int_M\Phi(y(\cdot,\bar\tau_2\tau_j),\bar\tau_2\tau_j)d\mu_{\bar\tau_2\tau_j}-\int_M\Phi(y(\cdot,\bar\tau_1\tau_j),\bar\tau_1\tau_j)d\mu_{\bar\tau_1\tau_j}
\\
\to\, &0\quad\text{ as} \quad j\to \infty,
\end{align*}
where we have used the boundedness (Lemma \ref{integralalwaysbound}) and the monotonicity (Theorem \ref{Huiskenmonotonicity}) of $\int_M\Phi(y(\cdot,\tau),\tau)d\mu_{\tau}$.  Since
$$\left|\vec{\tilde {H_j}}+\frac{\tilde y_j^\perp(\cdot,\tau)}{2\tau}\right|^2\Phi(\tilde y_j(\cdot,\tau),\tau)d\tilde\mu_{j,\tau}\to \left|\vec{\tilde {H}}_\infty+\frac{\tilde y_\infty^\perp(\cdot,\tau)}{2\tau}\right|^2\Phi(\tilde y_\infty(\cdot,\tau),\tau)d\tilde\mu_{\infty,\tau}$$
pointwise on $M\times[1,\alpha^2]$, Fatou's lemma implies
\begin{align*}
&\int_{\bar\tau_1}^{\bar\tau_2}\int_M\left|\vec{\tilde {H}}_\infty+\frac{\tilde y_\infty^\perp(\cdot,\tau)}{2\tau}\right|^2\Phi(\tilde y_\infty(\cdot,\tau),\tau)d\tilde\mu_{\infty,\tau}d\tau
\\
\leq\,&\liminf_{j\rightarrow\infty}\int_{\bar\tau_1}^{\bar\tau_2}\int_M\left|\vec{\tilde {H}}_j+\frac{\tilde y_j^\perp(\cdot,\tau)}{2\tau}\right|^2\Phi(\tilde y_j(\cdot,\tau),\tau)d\tilde\mu_{j,\tau}d\tau
\\
\leq\, &\liminf_{j\rightarrow\infty} \left(\int_M\Phi(\tilde y_j(\cdot,\bar\tau_2),\bar\tau_2)d\tilde\mu_{j,\bar\tau_2}-\int_M\Phi(\tilde y_j(\cdot,\bar\tau_1),\bar\tau_1)d\tilde\mu_{j,\bar\tau_1}\right)
\\
=\,&0.
\end{align*}
Hence, $\tilde y_\infty$ satisfies $$\vec{\tilde {H}}_\infty+\frac{\tilde y_\infty^\perp(\cdot,\tau)}{2\tau}=0.$$ This finishes the proof of Theorem  Theorem \ref{main}.

\begin{rem}\label{remarkmm}
	The no shrinking breather theorem in the closed case is proved in \cite[Theorem 2.3]{MM} by using the supremum of Huisken's functional. 
	For the sake of completeness, we give their proof below.
	 For a closed mean curvature flow $x:M^n\times I\rightarrow\mathbb{R}^{n+m}$, let us define
	$$
	H_{x_{0}, T}\left(M_t\right)=\int_M\frac{1}{(4\pi T)^{-\frac{n}{2}}}\exp\left(-\frac{|x(\cdot,t)-x_0|^2}{4T}\right)d\mu_t,\quad t\in  I,
	$$	
	and
	\begin{equation}
\lambda_{t_0}(t)=\sup\limits_{x_0}H_{x_{0},t_0-t}\left(M_t\right),\quad t\in(-\infty,t_0)\cap I.
	\end{equation}
	Since $M$ is closed, $H_{x_{0}, T}\left(M_t\right)\to 0$ for as $x_0\to \infty$. Hence $\lambda_{t_0}(t)$ can be achieved for any $t$. Moreover,	
	Huisken's monotonicity formula implies that $\lambda_{t_0}(t)$ is finite and is monotonically non-increasing in $t$. If $x$ is also a shrinking breather, letting $t_1<t_2$, and $\alpha$ be as in the statement of Definition \ref{def_breather}, we may set $t_0=\frac{t_2-\alpha^{2}t_1}{1-\alpha^{2}}$. Then
	$t_0-t_2=\alpha^{2}(t_0-t_1)$. Assume $x_2\in\mathbb{R}^{n+m}$ is the point where $\lambda_{t_0}(t_2)$ is attained, that is, $\lambda_{t_0}(t_2)=H_{x_2,t_0-t_2}(M_{t_2})$, then
\begin{align*}
\lambda_{t_0}(t_2)&=H_{x_{2},t_0- t_2}\left(M_{t_2}\right)\le H_{x_{2},t_0- t_1}\left(M_{t_1}\right)=H_{x_{2},\alpha^{-2}(t_0-t_2)}\left(\alpha^{-1} F\circ M_{t_2}\right)\\
&=H_{F^{-1}(\alpha x_2),t_0-t_2}\left(  M_{t_2}\right)\le \lambda_{t_0}(t_2),
\end{align*}	 
where $F:\mathbb{R}^{n+m}\to\mathbb{R}^{n+m}$ is an isometry. It follows that $H_{x_{2},t_0- t_1}\left(M_{t_1}\right)=H_{x_{2},t_0- t_2}\left(M_{t_2}\right)$, and hence $x$ is a self-shinker.	
\end{rem}

\section{Proof of Theorem \ref{main3}}

After parabolic scaling and translating in time, we let $x_0:M^n\times[0,1]\to \mathbb{R}^{n+1}$ be the   expanding breather in the statement of Theorem \ref{main3}, where $0$ and $1$ correspond to $t_1$ and $t_2$ in Definition \ref{def_breather}, respectively. Then there exists $\alpha>1$, an isometry $F:\mathbb{R}^{n+1}\to \mathbb{R}^{n+1}$, and a diffeomorphism $\phi:M\to M$, such that
$$
x_0(p,1)=\alpha F\circ x_0(\phi(p),0)\quad\text{ for all }\quad p\in M.
$$
As before, we can find an orthogonal matrix $\mathcal R\in O(n+1)$ and a constant vector $V\in\mathbb{R}^{n+1}$, such that
\begin{equation}\label{breather_x}
	x_0(p,1)=\alpha F\big(  x_0(\phi(p),0)\big)=\alpha \mathcal R x_0(\phi(p),0)+V\quad\text{ for all }\quad p\in M.
\end{equation}

We follow the same idea as in the proof of Theorem \ref{main} and define an immortal solution. Precisely, for each $j\geq 0$, let
\begin{eqnarray}\label{x_j_definition}
	\displaystyle t_j&=&\sum^j_{k=0}\alpha^{2k},
	\\\nonumber
	x_j(p,t)&=&\alpha^{j}\mathcal R^j x_0\big(\phi^j(p),\alpha^{-2j}( t- t_{j-1})\big)+\sum\limits_{k=0}^{j-1}\alpha^{k}\mathcal R^k V,\ \  t\in [ t_{j-1}, t_j].
\end{eqnarray}
Obviously, we have $t_j\nearrow \infty$. By the same argument as in the proof of Lemma \ref{TypeIancientsolution}, we also have that the following spliced immortal solution is smooth.
\begin{eqnarray}\label{defined_immortal_solution}
	x(p ,t)=\left\{
	\begin{array}{ll}
		x_0(p ,t), \quad &  t\in [0,1], \\
		x_j(p ,t), \quad & t\in [t_{j-1}, t_j].
	\end{array}
	\right.
\end{eqnarray}

As in the proof of Theorem \ref{main}, we consider the rescaled sequence $\{\tilde x_j\}_{j=1}^\infty$ defined as
\begin{equation}\label{def_immortal}
\tilde{x}_j(p,t):= t_j^{-\frac{1}{2}}x\big(\phi^{-(j+1)}(p),t_j t\big)\quad\text{ for }\quad t\in \left[1,\tfrac{t_{j+1}}{t_j}\right].
\end{equation}
Then
\begin{equation}\label{estimate_11}
	\tilde{x}_j(p,t)=
	t_j^{-\frac{1}{2}}\alpha^{j+1}\mathcal R^{j+1} x_0\big(p,\alpha^{-2(j+1)}t_j(t-1)\big)+t_j^{-\frac{1}{2}}\sum\limits_{k=0}^{j}\alpha^{k}\mathcal R^k V.
\end{equation}
By (\ref{x_j_definition}), we have
\begin{align}\label{estimate_3}
	\left|t_j^{-\frac{1}{2}}\sum\limits_{k=0}^{j}\alpha^{k}\mathcal R^k V\right|&\le t_j^{-\frac{1}{2}}\sum\limits_{k=0}^{j}\alpha^{k} |V|\le c|V|,
\end{align}
where $c$ is a constant independent of $j$, and the scaling factor in \ref{estimate_11} is controlled, that is
\begin{align}\label{estimate_4}
	\alpha^{-2(j+1)}t_j\to c_0:=(\alpha^2-1)^{-1}.
\end{align}

From the computation above, we see that $\tilde x_j$ and $x_0$ differ only by a rotation, and a scaling with a controlled factor, and a transportation by a controlled vector. Hence, all the conditions of \cite[Theorem 11.12]{ACGL} are satisfied (it is obvious that this theorem holds true for noncompact mean curvature flows with higher codimension also, the reader can check this point easily). Note that the local area bound and the properness of immersion are both provided by the convexity assumption. Hence, we may obtain a mean curvature flow $\tilde x_\infty: M_\infty\times(1,\alpha^2]\to\mathbb{R}^{n+1}$ such that a subsequence of $\{\tilde x_j\}_{j=1}^\infty$ converges to $\tilde x_\infty$ on compact sets of $\mathbb{R}^{n+1}\times\mathbb{R}$ (see \cite[Definition 11.10]{ACGL}). By the compactness of $O(n+1)$, and by (\ref{estimate_3}), me may find $\mathcal R^\infty$ and $V_\infty$, such that, after passing to a further subsequence, we have 
$$\mathcal R^{j+1}\to \mathcal R^\infty\quad\text{ and }\quad t_j^{-\frac{1}{2}}\sum\limits_{k=0}^{j}\alpha^{k}\mathcal R^k V\to V_\infty.$$
Hence we have
$$M_\infty=M^n\quad\text{ and }\quad \tilde x_\infty(p,t)=c_0^{-\frac{1}{2}}\mathcal R_\infty x_0(p,c_0(t-1))+V_\infty.$$
It remains to show that $\tilde x_\infty$ is a self-expander.

In the argument below, if the notation is with tilde, then we are referring to the scaled mean curvature flow $\tilde x_j$, otherwise we are referring to the immortal solution $x$. Since
$x_0$ has bounded mean curvature $C$, we have
$$
\sup\limits_{p\in M}|H|(p,t)\le \alpha^{-(j+1)}C \le \frac{C}{\sqrt{t}}\quad\text{ for all }\quad t\in [t_j,t_{j+1}].
$$
It follows that
\begin{eqnarray}\label{ThetypeIIbbound}
\sup\limits_{p\in M}|H|(p,t) \le \frac{C}{\sqrt{t}}\quad\text{ for all }\quad t>0.
\end{eqnarray}
Let us fix a point $p_0\in M^n$. Then, for all $j\geq 1$, we have
$$\big|\,x(p_0,t)-x(p_0,1)\,\big|\leq \int_1^{t_{j+1}}|H(p_0,t)|dt\leq\int_1^{t_{j+1}}C t^{-\frac{1}{2}}\leq C t_j^{\frac{1}{2}}\quad\text{ for all }\quad t\in[t_j,t_{j+1}].$$
Hence, by (\ref{def_immortal}), we have
\begin{eqnarray*}
\big|\,\tilde x_j(\phi^{j+1}(p_0),t)\,\big|\leq C\quad\text{ for all }\quad t\in\left[1,\tfrac{t_{j+1}}{t_j}\right],
\end{eqnarray*}
where $C$ is a constant independent of $j$. It then follows from the definition of convergence on compact sets of $\mathbb{R}^{n+1}\times\mathbb{R}$ that, after passing to a further subsequence, we can find a point $p_\infty\in M$ satisfying
\begin{eqnarray}
\nabla^m\tilde x_j(\phi^{j+1}(p_0),t)\to\nabla^m\tilde x_\infty(p_\infty,t)\quad\text{for all}\quad m\in\mathbb{N}\quad\text{ and }\quad t\in(1,\alpha^2].
\end{eqnarray}

Recall Hamilton's matrix Harnack estimate for weak convex mean curvature flow,
\begin{eqnarray}\label{harnack}
\frac{\partial H}{\partial t}+\frac{H}{2t}+2V_i\nabla_i H+h_{ij}V_iV_j\geq 0,
\end{eqnarray}
where $V$ is any vector field on $M$. From Hamilton's Harnack estimate, we have that  $\sqrt{t}H(p_0,t)$ is non-decreasing. Furthermore, by (\ref{ThetypeIIbbound}), $\sqrt{t}H(p_0,t)$ is bounded from above. Hence we may compute
$$\sqrt{t}\tilde H_{\infty}(p_\infty,t)=\lim\limits_{j\to\infty}\sqrt{t}\tilde H_j(\phi^{j+1}(p_0),t)=\lim\limits_{j\to\infty}\sqrt{t_jt}H(p_0,t_jt)=\operatorname{const}\quad\text{ for }\quad t\in(1,\alpha^2].$$
In other words, we have
\begin{eqnarray}\label{harnackcritical}
\frac{\partial}{\partial t}\left(\sqrt{t}\tilde H_\infty(p_\infty,t)\right)=0\quad\text{ for all }\quad t\in(1,\alpha^2].
\end{eqnarray}
Without loss of generality, we assume that $\tilde x_{\infty}$ is strictly convex, for otherwise, by the strong maximum principle,
$M_{\infty}$ splits as $N_{\infty}^k\times \mathbb{R}^{n-k}$ with
$N_{\infty}^k$ being strictly covex, and we may work with $N_{\infty}^k$ instead. Taking $V_j=-(\tilde{h}_{\infty})^{-1}_{ ij}\nabla_j\tilde{H}_{\infty}$ in (\ref{harnack}), we get $\nabla\tilde H_\infty(p_\infty,t)=0$.
It then follows from \cite[Theorem A.2]{C} and (\ref{harnackcritical}) that $\tilde x_\infty$ is a self-expander. This finishes the proof of Theorem \ref{main3}.

\section{Proof of Theorem \ref{main5}}

The proof of Theorem \ref{main5} is similar to that of Theorem \ref{main3}. We first of all construct an ancient solution out of the steady breather.
Let $x_0:M^n\times [0,1]\to \mathbb{R}^{n+m}$ be the  steady  breather as described in the statement of Theorem \ref{main5}. Let $F:\mathbb{R}^{n+1}\to\mathbb{R}^{n+1}$ and $\phi:M\to M$ be the isometry and self-diffeomorphism in Definition \ref{def_breather}, respectively. Then we have
\begin{equation}\label{the_steady_breather}
	x_0(p,1)= F\circ x_0(\phi(p),0)\quad \text{ for all }\quad p\in M.
\end{equation}
Let us define
\begin{eqnarray}\label{translationinj}
x_{-j}(p,t)=F^{-j}\circ x_0\big(\phi^{-j}(p),t+j\big)\quad\text{ for }\quad t\in[-j,-j+1]
\end{eqnarray}
and the ancient solution
\begin{eqnarray}\label{ancientsteady}
x:M^n\times(-\infty,1]\rightarrow \mathbb{R}^{n+1},\quad \quad x(\cdot,t)=x_{-j}(\cdot,t)\quad\text{ for }\quad t\in[-j,-j+1].
\end{eqnarray}
We then consider the two cases in Theorem \ref{main5}.
\\

\noindent (1) We assume $\max_M H(\cdot,0)$ is attained at $p_0$, here $H$ is the mean curvature of $x_0(\cdot,0)$. We shall also use $H$ to denote the mean curvature of $x$, since this does not cause any ambiguity. By (\ref{the_steady_breather}), since $F$ is an isometry, we have that $H(\cdot,1)=H(\phi(\cdot),0)$. It then follows that 
$$\max_M H(\cdot,1)= H(\phi^{-1}(p_0),1)=\max_M H(\cdot,0)=H(p_0,0).$$
On the other hand, Hamilton's Harnack \cite{RH} implies that
\begin{equation}\label{ancient_harnack}
\frac{\partial H}{\partial t}+2V_i\nabla_iH+h_{ij}V_iV_j\geq0,
\end{equation}
where $V$ is any vector field on $M$. Consequently we have $\partial_t H(\cdot,t)\geq 0$, and hence
\begin{gather*}
H(p_0,0)\leq H(p_0,1)\leq \max_{M}H(\cdot,1)=H(p_0,0),
\\
\sup_{M\times(-\infty,1]}H\leq \max_{M}H(\cdot,1).
\end{gather*}
It follows that the space-time maximum of $H$ on $M\times(-\infty,1)$ is attained at $(p_0,0)$. We may without loss of generality assume that $x$ is strictly convex, for otherwise it splits as $N^{n-k}\times \mathbb{R}^k$ by the strong maximum principle, and we can consider $N^{n-k}$ instead. Then the computation in \cite[Theorem B]{RH} implies that $x$ is a translator.
\\

\noindent (2) In our proof, we shall consider the case where the steady breather satisfies
\begin{eqnarray}\label{bound_condition_1}
\sup_{j\in\mathbb{N}}\left|\,  x(\phi^{-j}(p_0),0) \, \right|<\infty,
\end{eqnarray}
since the other case is parallel, and will be briefly described at the end of the proof. Since $F:\mathbb{R}^{n+1}\to\mathbb{R}^{n+1}$ in (\ref{the_steady_breather}) is an isometry, we can find $\mathcal R\in O(n+1)$ and $V\in\mathbb{R}^{n+1}$, such that (\ref{the_steady_breather}) becomes
\begin{equation}\label{the_steady_breather_2}
	x_0(p,1)= \mathcal R x_0(\phi(p),0)+V\quad \text{ for all }\quad p\in M.
\end{equation}
Consequently, (\ref{translationinj}) becomes
\begin{eqnarray}\label{somethingbackward}
x_{-j}(p,t)=\mathcal R^{-j}x_0\big(\phi^{-j}(p),t+j\big)-\sum_{k=1}^{j}\mathcal R^{-k}V\quad\text{ for }\quad t\in[-j,-j+1].
\end{eqnarray}

We shall consider the ancient solution defined in (\ref{ancientsteady}). Denoting $\tilde x_j(p,t):=x\big(\phi^j(p),t-j\big)+\sum_{k=1}^{j}\mathcal R^{-k}V$ for $t\in[-1,1]$, then we have
\begin{align*}
\tilde x_j(p,t)=\left\{\begin{array}{ll}
\mathcal R^{-j}x_0(p,t), & t\in[0,1],
\\
\mathcal R^{-j-1}x_0\big(\phi^{-1}(p),t+1\big)-\mathcal R^{-j-1}V,& t\in[-1,0].
\end{array}\right.
\end{align*}
Since $|\mathcal R^{-j-1} V|\equiv |V|$ for all $j\le 0$, we may also apply \cite[Theorem 11.12]{ACGL} as we have done in the proof of Theorem \ref{main3}. After passing to a subsequence, the sequence of mean curvature flows $\{\tilde x_j\}_{j=1}^\infty$ converges to $\tilde x_\infty:M^n_\infty\times (-1,1]\rightarrow\mathbb{R}^{n+1}$ on compact sets of $\mathbb{R}^{n+1}\times\mathbb{R}$ (see \cite[Definition 11.10]{ACGL}), with
\begin{gather*}
M_\infty=M,
\\
\tilde x_\infty(\cdot,t) =\mathcal R^\infty x_0(\cdot,t)\quad\text{ for all }\quad t\in[0,1],
\end{gather*}
for some $\mathcal R^\infty\in O(n+1)$. It remains to show that $\tilde x_\infty$ is a translator.

By our assumption (\ref{bound_condition_1}), we have $$\sup_j\left|\,\tilde x_j(\phi^{-j}(p_0),0)\,\right|<\infty.$$ Then we can find a $p_\infty\in M$, such that $\nabla^m\tilde x_j(\phi^{-j}(p_0),0)\to \nabla^m\tilde x_\infty(p_\infty,0)$ for all $m\in\mathbb{N}$, and therefore
\begin{eqnarray*}
\frac{\partial}{\partial t}\tilde H_\infty(p_\infty,0)=\lim_{j\to\infty}\frac{\partial}{\partial t}\tilde H_j(\phi^{-j}(p_0),0)=\lim_{j\to\infty}\frac{\partial}{\partial t} H(p_0,-j)=0,
\end{eqnarray*}
where the last equality is because of Hamilton's Harnack estimate $\partial_t H(p_0,t)\geq 0$ for the ancient solution defined in (\ref{ancientsteady}) and the fact that $H\geq 0$.
 We may without loss of generality assume that $\tilde x_\infty$ is strictly convex as above. Taking $V_j=-(\tilde{h}_{\infty})^{-1}_{ ij}\nabla_j\tilde{H}_{\infty}$ in (\ref{ancient_harnack}), we get $\nabla\tilde H_\infty(p_\infty,0)=0$.
 The rest of the proof is not different from that of Theorem \ref{main3}; the computation is the same as \cite{RH}. This finishes the proof of Theorem \ref{main5}.

Concerning the case where 
\begin{eqnarray}\label{bound_condition_2}
\sup_{j\in\mathbb{N}}\left|\,  x(\phi^{j}(p_0),0) \, \right|<\infty
\end{eqnarray}
 holds, we may define 
\begin{eqnarray*}
x_{j}(p,t)=F^{j}\circ x_0\big(\phi^{j}(p),t-j\big)\quad\text{ for }\quad t\in[j,j+1],
\end{eqnarray*}
and consider the eternal solution
\begin{eqnarray*}
x(p,t)=\left\{\begin{array}{ll}
x_{-j}(p,t) & t\in[-j,-j+1],
\\
x_{j}(p,t) & t\in[j,j+1],
\end{array}\right.
\end{eqnarray*}
for all $j\in\mathbb{N}$, where $x_{-j}$ is defined in (\ref{translationinj}) and (\ref{somethingbackward}). Obviously, on this eternal solution Hamilton's Harnack is still true. Then, writing 
$$x_j(p,t)=\mathcal R^jx_0(\phi^j(p),t-j)+\sum_{k=0}^{j-1}\mathcal R^k V\quad\text{ for }\quad t\in[j,j+1],$$
we may consider the sequence of mean curvature flows $\tilde x_j: M^n\times[-1,1]\to\mathbb{R}^{n+m}$, defined by $\tilde x_j(p,t)=x(\phi^{-j}(p),t+j)-\sum_{k=0}^{j-1}\mathcal R^k V$, that is,
\begin{eqnarray*}
\tilde x_j(p,t)=\left\{\begin{array}{ll}
\mathcal R^j x(p,t) & t\in[0,1],
\\
\mathcal R^{j-1} x(\phi^{-1}(p),t+1)-\mathcal R^{j-1}V & t\in[-1,0].
\end{array}\right.
\end{eqnarray*}
We may then apply the same argument to the sequence $\{\tilde x_j\}_{j=1}^\infty$ as before and use the condition (\ref{bound_condition_2}) in place of (\ref{bound_condition_1}). The conclusion follows similarly.

\section{Proof of Theorem \ref{Coro_main5}}
We show that there exists no nontrivial weakly mean convex rotator. Arguing by contradiction, we assume that $x_0:M^n\to\mathbb{R}^{n+1}$ is a weakly mean convex rotator, satisfying 
\begin{align}\label{rotator}
H=-\langle Jx_0,n\rangle
\end{align}
for some $J\in\mathfrak{so}(n+1)$. Obviously, $x_0$ cannot be closed, since otherwise it generates a closed immortal solution $x:M^n\times\mathbb{R}\to\mathbb{R}^{n+1}$, which does not exist.

Assume that $x_0$ is noncompact. Since it is properly embedded, we have $x_0(p)\to\infty$ whenever $p\to\infty$. Hence, we can find a point $p_0\in M$, such that $|x_0|$ attains its minimum at $p_0$. If $x_0(p_0)$ is the origin of $\mathbb{R}^{n+1}$, then by (\ref{rotator}) we immediately have that $H(p_0)=0$. If $x_0(p_0)$ is not the origin of $\mathbb{R}^{n+1}$, then at $p_0$ we compute
$$2x_0^\top=2Dx_0\cdot x_0=D|x_0|^2=0.$$
Hence $n(p_0)=cx_0(p_0)$ for some constant $c$. Since $J=-J^T$, by (\ref{rotator}) we have
$$H(p_0)=-\langle Jx_0(p_0),n(p_0)\rangle=-c\langle Jx_0(p_0),x_0(p_0)\rangle=0.$$
The strong maximum principle immediately implies that $H\equiv 0$ on the mean curvature flow $x$ generated by $x_0$. If $x_0$ is weakly convex, we also have that $x_0$ is totally geodesic, and hence it is a hyperplane.

\begin{rem}
We remark that any minimal surface can be view as a trivial rotator if we 
take $J=0$ in (\ref{rotator}).
\end{rem}

\end{document}